\documentclass{amsart}
\usepackage{setspace}
\usepackage{times}
\usepackage{graphicx}
\usepackage{wrapfig}
\usepackage{amsmath, amsthm, amssymb, amsfonts}
\usepackage{enumitem}
\usepackage{caption}
\usepackage{subcaption}
\usepackage{bbm}
\usepackage{mathtools}
\usepackage{xypic}
\usepackage{dsfont}

\everymath{\displaystyle}
\newcommand{\abs}[1]{\left| #1 \right|}

\newtheorem{thm}{Theorem}
\newtheorem*{thm*}{Theorem}

\newtheorem{lem}{Lemma}
\newtheorem*{lem*}{Lemma}
\newtheorem{prop}{Proposition}

\theoremstyle{remark}

\theoremstyle{definition}
\newtheorem*{defin*}{Definition}
\newtheorem{defin}{Definition}

\setenumerate{topsep=-\baselineskip}

\newcommand{\norm}[1]{\left\lVert#1\right\rVert}
\renewcommand{\epsilon}{\varepsilon}
\newcommand{\bbe}{\mathbb{E}}

\title{A Note on Strongly Mixing Extensions}
\author{Mike Schnurr}
\address{Max Planck Institute for Mathematics in the Sciences, Inselstr. 22, 04103 Leipzig, Germany}
\email{schnurr@mis.mpg.de}

\begin{document}
	
\begin{abstract}
	Continuing the investigations by the author \cite{SchnurrWM} and Glasner and Weiss \cite{GlasnerWeiss} on generic properties of extensions, we give a sufficient condition for the strongly mixing extensions of a fixed transformation to be of first category.
\end{abstract}

\maketitle

\section{Introduction}
The study of generic properties of measure-preserving transformations was initiated by Halmos \cite{HalmosPaper} and Rokhlin \cite{Rokhlin}, see e.g. \cite{Nadkarni}, \cite{KatokStepin}, \cite{Ageev1}, \cite{King}, \cite{Ageev2}, \cite{AlpernPrasad}, \cite{Ageev3}, \cite{RueLazaro}, \cite{Ageev4}, \cite{Solecki}, \cite{Guiheneuf}. In particular, a generic measure-preserving transformation is rigid \cite{KatokStepin} and weakly mixing \cite{HalmosPaper}. One of the important tools was the Halmos Conjugacy Lemma.

\begin{lem}[Halmos Conjugacy Lemma] \label{Lem:ConjugacyLemma}
	Let $(X,\mu)$ be a measure space and let $T$ be an (invertible) antiperiodic measure-preserving transformation on $(X,\mu)$. Then 
	
	$$\{S^{-1}TS| S \text{ is an invertible measure-preserving transformation} \}$$ 
	
	is dense in the weak topology on the set of (invertible) measure preserving transformations on $X$. 
\end{lem}

Surprisingly, the study of generic properties of extensions began in 2017, originating from a question by Terence Tao, asking if a generic extension does not have a nontrivial intermediate nilfactor. While it was shown in \cite{SchnurrWM} that in the case of a non-fixed factor, a generic extension is relatively weakly mixing, but not relatively mixing, an analogue of the Halmos Conjugacy Lemma for the fixed factor case was done in \cite{GlasnerWeiss}. As a consequence, a generic extension is relatively weakly mixing.

In this note we give a sufficient condition for the strongly mixing extensions of a fixed transformation $T_0$ to form a first category set. We then show that if $T_0$ is either periodic, or rigid and antiperiodic, this condition is met. We then give a possible definition of rigid extensions and show that for any factor $T_0$, the rigid extensions of $T_0$ are $G_{\delta}$.

\textbf{Acknowledgments.} The author expresses gratitude to Jean-Paul Thouvenot for the inspiration that resulted in this note, as well as to Tanja Eisner and Philipp Kunde for helpful discussions. The support of the Max Planck Institute is also greatly appreciated.
	
\section{Notation and Definitions}	
Let $(X,m)$ be the unit interval with the Lebesgue measure, and $\mathcal{G}(X)$ denote the set of invertible measure-preserving transformations on $X$. We will also have $(Y,\nu)$ denote the unit interval with the Lebesgue measure. Let $(Z,\mu):= (X \times Y, m \times \nu).$ For $T_0 \in \mathcal{G}(X),$ denote by $\mathcal{G}_{T_0}$ the subset of $\mathcal{G}(Z)$ which are extensions of $T_0$ (through the natural projection onto $X$). Note that we will also use $T_0$ to represent the transformation $T_0 \times \mathds{1}_Y$. The topology on $\mathcal{G}(X)$ and $\mathcal{G}_{T_0}$ will always be assumed to be the weak topology. For $T\in \mathcal{G}_{T_0},$ let $\{T_x\}_{x \in X} \subset \mathcal{G}(Y)$ be such that $T(x,y)=(T_0x,T_xy).$

For $p \in \mathbb{N}, p >1,$ a $p$-adic interval of rank $k$ in $Y$ is an interval of the form 

$$\left(\frac{j}{p^k},\frac{j+1}{p^k}\right), 0 \le j < p^k,$$

and a $p$-adic set is a union of $p$-adic sets. A $p$-adic permutation (of rank $k$) is a permutatiion of the $p$-adic intervals of rank $k$.

\begin{lem}
	For all $p>1,$ the $p$-adic permuations are dense in $\mathcal{G}(Y).$
\end{lem}

The proof is nearly identical to the one found in \cite{HalmosLectures} in the proof of the Weak Approximation Theorem (p.66); indeed, up to replacing each instance of dyadic with $p$-adic, the proofs are identical.

Next we recall a couple of definitions. First, we say $T \in \mathcal{G}(X)$ is \emph{antiperiodic} if the set of periodic points for $T$ has measure zero. 

\begin{defin} \label{Def:Rigid}
	We say that $T \in \mathcal{G}(X)$ is \emph{rigid} if there exists a subsequence $\{n_k\}$ such that $T^{n_k} \to \mathds{1}_X$ strongly.
\end{defin}

Let $\mathcal{R}(X) \subset{G}(X)$ denote the set of rigid transformations on $X$. The following equivalent definition is well known.

\begin{prop} \label{Prop:RigidEquivalence}
	A measure-preserving transformation $T \in \mathcal{G}(X)$ is rigid if and only if there exists a subsequence $\{n_k\}$ such that for all $f,g \in L^2(X),$
	
	\begin{equation} \label{Eqn:Rigidity}
	\lim_{k \to \infty} \abs{\langle T^{n_k}f,g \rangle - \langle f,g \rangle}=0.
	\end{equation}
	
\end{prop}

\begin{proof}
	Equation (\ref{Eqn:Rigidity}) says that $T^{n_k} \to \mathds{1}_X$ weakly, so one direction is trivial. For the other direction, suppose there is a subsequence $\{n_k\}$ such that (\ref{Eqn:Rigidity}) holds. As $T$ is  unitary, clearly $\norm{T^{n_k}f} \to \norm{\mathds{1}_Xf}$ for all $f$, so $T^{n_k} \to \mathds{1}_X$ strongly.
\end{proof}

The next definition can be found in \cite{SchnurrWM}.

\begin{defin} \label{Def:StronglyMixingExtension}
	We say that $T \in \mathcal{G}_{T_0}$ is a \textit{(strongly) mixing extension} of $T_0$ or $T$ is \textit{(strongly) mixing relative to} $T_0$ if for all $f,g \in L^2(Z|X),$
	
	$$\lim_{n \to \infty} \norm{\bbe(T^nf \cdot \overline{g} |X) - T_0^n \bbe(f|X)\bbe(\overline{g}|X)}_{L^2(X)} = 0.$$
\end{defin}

	Let $\mathcal{S}_{T_0} \subset \mathcal{G}_{T_0}$ denote the set of strongly mixing extensions of $T_0.$ Lastly for definition reminders, we have the following, from \cite{GlasnerWeiss}.

\begin{defin} \label{Def:PiecewiseConstant}
	Let $T \in \mathcal{G}_{T_0}$. We say that $T$ is \emph{a piecewise constant skew product} over $T_0$ if there exists a partition of $X, \{A_1, \ldots, A_N \},$ and $\{R_1, \ldots, R_N \} \subset \mathcal{G}(Y)$ such that for $1 \le j \le N$ and for all $x \in A_j, T_x = R_j$.
\end{defin}

Lemma 1.3 in \cite{GlasnerWeiss} shows that the set of piecewise constanst skew products over $T_0$ is dense in $\mathcal{G}_{T_0}.$

\section{Stronly Mixing Extensions}
To fit the classical theory, we would like to prove that for all $T_0, \mathcal{S}_{T_0}$ is a first category subset of $\mathcal{G}_{T_0}$. Unfortunately, this goal proves to be more difficult than one would expect. We will instead be able to show that for a set of $T_0$ which is generic in $\mathcal{G}(X), \mathcal{S}_{T_0}$ is of first category. First we will give the Glasner and Weiss \cite{GlasnerWeiss} version of the Halmos Conjugacy Lemma. While the result and its proof (including much of their notation) can be found in their note, we present it with slightly more generality.

\begin{thm} \label{Thm:ConjugacyFixed}
	Let $T_0 \in \mathcal{G}(X)$ be antiperiodic, and $\hat{T} \in \mathcal{G}_{T_0}$. Then $\{S^{-1}\hat{T}S| S \in \mathcal{G}_{\mathds{1}_X} \}$ is dense in $\mathcal{G}_{T_0}.$
\end{thm}

\begin{proof}
	Let $T \in \mathcal{G}_{T_0}$ and let $\epsilon > 0$, and $N_{\epsilon}(T)$ a dyadic neighborhood of $T$. We wish to find $S$ such that $S^{-1} \hat{T}S \in N_{\epsilon}(T)$. Without loss of generality, we can assume that $T$ is a piecewise constant skew product over $T_0$. Let $\{A_1, \ldots, A_k \}$ be a partition of $X$ and $\{R_1, \ldots, R_k \} \subset \mathcal{G}(Y)$ such that for all $j$, and all $x \in A_j,T_x=R_j.$
	
	Now let $\{B, T_0B, \ldots, T_0^{n-1}B \}$ be a Rokhlin tower with respect to $\epsilon/2,$ and $1/n < \epsilon/2$, that is, $B, T_0B, \ldots,$ $T_0^{n-1}B$ are pairwise disjoint, and $m\left(\bigcup_{i=0}^{n-1}T_0^iB\right) > 1- \epsilon$. We now refine the tower with respect to the partition of $X, \{A_j| 1 \le j \le k\}$. This will create a partition of $B$ into sets $\{B_l | 1 \le l \le L\}$ such that for all $i, l,$ there exists $\alpha(l,i)$ between $1$ and $k$ such that $T_0^iB_l \subset A_{\alpha(l,i)}$.
	
	We will now define $S \in \mathcal{G}_{\mathds{1}_X},$ a piecewise constant skew product. First, for all $x \in B$ define $S_x := \mathds{1}_Y.$ Then, inductively, having defined $S$ for $0 \le i < n-1$ for all $l,$ define $S_x$ for $x \in T^{i+1}B_l$ (for all $l$) to be $S_x := \hat{T}_{T_0^{-1}x}S_{T_0^{-1}x}R^{-1}_{\alpha(l,i)}$. Note that this implies $S^{-1}_x\hat{T}_{T_0^{-1}x}S_{T_0^{-1}x} = R_{\alpha(l,i)}.$ This defines $S$ on the fibers corresponding to the Rokhlin tower. Finally, if $x \notin \bigcup_{i=0}^{n-1}T_0^iB,$ define $S_x:= \mathds{1}_Y$.
	
	We see that by how we constructed $S$, if $x \in \bigcup_{i=0}^{n-2}T_0^iB$ then $T_x := R_x= (S^{-1}\hat{T}S)_x.$ Thus it is easy to see that  for any measurable set $E \in Z,$
	
	$$\mu(TE \triangle S^{-1} \hat{T} S E) \le m\left(X \backslash \bigcup_{i=0}^{n-1}T_0^iB\right) + m(T^{n-1}B) < \frac{\epsilon}{2} + \frac{1}{n} < \epsilon,$$
	
	so certainly $S^{-1}\hat{T}S \in N_{\epsilon}(T).$
\end{proof}

Next, we give a sufficient condition for $\mathcal{S}_{T_0}$ to be of first category

\begin{lem} \label{Lem:RigidFactorImpliesRare}
	Let $T_0 \in \mathcal{G}(X)$ be fixed. If $\mathcal{R}(Z) \cap \mathcal{G}_{T_0}$ is dense in $\mathcal{G}_{T_0},$ then $\mathcal{S}_{T_0}$ is a first category subset of $\mathcal{G}_{T_0}.$ 
\end{lem}

\begin{proof}
	Let $A := X \times [0, 1/2]$ and for all $k \in \mathcal{N}$ define
	
	$$P'_k := \left\{T \in \mathcal{G}_{T_0} | \mu(T^kA \cap A) > \frac{9}{20} \right\}.$$
	
	Note that for all $n, P_n := \bigcup_{k > n} P'_k$ is dense in $\mathcal{G}_{T_0}$ as $\mathcal{R}(Z) \cap \mathcal{G}_{T_0} \subset P_n$ for all $n$.
	
	Now define
	
	$$M'_k := \left\{T \in \mathcal{G}_{T_0} | \norm{\bbe(T^k \chi_A \cdot \chi_A | X) - \frac{1}{4}}_{L^2(X)} \le \frac{1}{5} \right\},$$
	
	and $M_n := \bigcap_{k > n} M'_k$. Note that as $\bbe(\chi_A |X) \equiv 1/2, \mathcal{S}_{T_0} \subset \bigcup_n M_n$. Further, as $M'_k$ is closed (see Theorem 6 in \cite{SchnurrWM}. Alternatively, see Proposition \ref{Prop:REGdelta} below), to prove $\mathcal{S}_{T_0}$ is of first category, it is sufficient to prove $M_n$ is dense for all $n$. But this is true as $P'_k \subset (M'_k)^c$ for all $k$. Indeed, if $T \in P'_k,$ then 
	
	\begin{align*}
	& \quad \ \norm{\bbe(T^k \chi_A \cdot \chi_A |X) - 1/4}_{L^2(X)} \ge \norm{\bbe(T^k \chi_A \cdot \chi_A |X) - 1/4}_{L^1(X)} \\
	&= \norm{T^k \chi_A \cdot \chi_A }_{L^2(Z)} - 1/4 > 9/20 - 1/4 = 1/5.
	\end{align*}
\end{proof}

While Lemma \ref{Lem:RigidFactorImpliesRare} is a nice sufficient condition, it certainly cannot be applied for all factors $T_0$. Indeed, if $T_0$ is strongly mixing, and if $T \in \mathcal{G}_{T_0}$ and $B$ is any $X$-measurable set, then for no subsequence ${n_k}$ does $T^{n_k} \chi_B \to \chi_B$. All the same, it will be an invaluable tool, as we will see in the following two results.

\begin{thm} \label{Thm:PeriodicFactorDense}
	If $T_0$ is periodic, then $\mathcal{S}_{T_0}$ is a first category subset of $\mathcal{G}_{T_0}$.
\end{thm}

\begin{proof}
	By Lemma \ref{Lem:RigidFactorImpliesRare}, we need only show that $\mathcal{R}(Z) \cap \mathcal{G}_{T_0}$ is dense in $\mathcal{G}_{T_0}$. Let $p$ be the smallest integer greater than 1 such that $T_0^p = I_X.$ Further fix $T \in \mathcal{G}_{T_0}, \epsilon > 0,$ and let $N_{\epsilon}(T)$ be a $p$-adic neighborhood in the weak topology. That is,
	
	$$N_{\epsilon}(T) := \{S \in \mathcal{G}_{T_0} | \mu(TD_{ij} \triangle SD_{ij}) < \epsilon, 1 \le i,j \le p^K \}$$
	
	where $D_{ij}$ are (without loss of generality) every $p$-adic square of rank $K \in \mathbb{N}$. Let $\{E_i\}, \{F_i\}$ be dyadic intervals of rank $K$ in $X$ and $Y$ respectively such that $D_{ij} = E_i \times F_j.$ 
	
	By Lemma 1.3 of \cite{GlasnerWeiss}, there exists $S \in N_{\epsilon/2}(T)$ that is a piecewise constant skew product over $T_0$. Let $\{A_1, \ldots A_N\}$ partition $X$ and $\{R_1, \ldots R_N \} \subset \mathcal{G}(Y)$ such that for all $x \in A_k, S_x = R_k.$
	
	Now for $V \in \mathcal{G}(Y),$ let 
	
	$$N'_{\epsilon}(V) := \{U \in \mathcal{G}(Y) | \nu(VF_i \triangle UF_i) < \epsilon, 1 \le i \le p^K \},$$
	
	and for each $R_k,$ let $P_k \in N'_{\epsilon/(2Nm(A_k))}(R_k)$ be a $p$-adic permutation. Define a new piecewise constant extension of $T_0, Q$ such that for all $x \in A_k, Q_x = P_k$. Fix $D_{ij}$. Then by Fubini,
	
	\begin{align*}
	&\quad \ \mu(SD_{ij} \triangle QD_{ij}) = \int_{E_i} \nu(S_xF_j \triangle Q_xF_j) dx \\
	&= \sum_{k=1}^N \nu(R_kF_j \triangle P_kF_j) m(A_k \cap E_i) < \sum_{k=1}^N \frac{\epsilon}{2N m(A_k)} m(A_k \cap E_i) \le \frac{\epsilon}{2}.
	\end{align*}
	
	Thus, $\mu(TD_{ij} \triangle QD_{ij}) \le \mu(TD_{ij} \triangle SD_{ij}) + \mu(SD_{ij} \triangle QD_{ij}) < \epsilon.$ As this holds for all $i,j, Q \in N_{\epsilon}(T).$ Let $M$ be the maximum of the ranks of all $P_k.$ Then $Q^{p^{M+1}}=I_Z,$ so $Q \in \mathcal{R}(Z) \cap \mathcal{G}_{T_0}$.
\end{proof}

\begin{thm} \label{Thm:AntiperiodicanRigidImpliesRare}
	Let $T_0 \in \mathcal{G}(X)$ be antiperiodic and rigid. Then $\mathcal{S}_{T_0}$ is a first category subset of $\mathcal{G}_{T_0}.$
\end{thm}

\begin{proof}
	Let $T \in \mathcal{R}(Z) \cap \mathcal{G}_{T_0}.$ We know this set is nonempty as the rigidity of $T_0$ on $X$ implies the rigidity of $T_0 \times \mathds{1}_Y$ on $Z$. Now consider the set
	
	$$\{S^{-1}TS | S \in \mathcal{G}_{\mathds{1}_X} \}.$$
	
	By Theorem \ref{Thm:ConjugacyFixed}, this set is dense in $\mathcal{G}_{T_0}.$ Further, as $S$ can be viewed as an element of $\mathcal{G}(Z),$ we know that for each $S, S^{-1}TS \in \mathcal{R}(Z).$ Thus by Lemma \ref{Lem:RigidFactorImpliesRare}, $\mathcal{S}_{T_0}$ is of first category.
\end{proof}

To close this section we show that, though it is not completely clear for exactly which set of $T_0 \mathcal{S}_{T_0}$ is of first category, we are able to show a step in that direction.

\begin{prop} \label{Prop:SMECodense}
	Let $T_0 \in \mathcal{G}(X)$ be antiperiodic. Then $\mathcal{S}_{T_0}^c$ is dense in $\mathcal{G}_{T_0}.$
\end{prop}

\begin{proof}
	Let $\mathcal{A}$ denote the measurable subsets of $Z$ such that $A \in \mathcal{A}$ if and only if $\bbe(\chi_A |X) \equiv 1/2$. Further, define
	
	$$\mathcal{A}_{T_0} := \left\{T \in \mathcal{G}_{T_0} | \  \exists \ \{n_k\}, A \in \mathcal{A} \ s.t. \ \norm{\bbe(T^{n_k}\chi_A \cdot \chi_A |X)- \frac{1}{4}} \ge \frac{1}{5} \ \forall k \right\}.$$

	Note that for all $T_0, \mathcal{A}_{T_0}$ is nonempty as for $T = T_0 \times \mathds{1}_Y$ we can take $A$ as in the proof of Lemma \ref{Lem:RigidFactorImpliesRare}. Further note that $\mathcal{A}_{T_0} \subset \mathcal{S}_{T_0}^c$. Now fix $T \in \mathcal{A}_{T_0},$ as well as $A \in \mathcal{A},$ and $\{n_k\}$ corresponding to $T$. Now fix $S \in \mathcal{G}_{\mathds{1}_X}$ and define $f := S^{-1}\chi_A.$ Then we have

	$$\bbe((S^{-1}TS)^{n_k}f \cdot f |X)= \bbe(T^{n_k}(Sf) \cdot (Sf)|X) = \bbe(T^{n_k}\chi_A \cdot \chi_A |X).$$

	Thus we see that $S^{-1}TS \in \mathcal{A}_{T_0}.$ By Theorem \ref{Thm:ConjugacyFixed}, $\mathcal{A}_{T_0}$ is dense in $\mathcal{G}_{T_0}.$
\end{proof}

\section{Rigid Extensions}
In \cite{KatokStepin}, Katok and Stepin showed a generalization of Rokhlin's result that strongly mixing transformations are of first category, namely that rigid transformations are residual. Indeed, if a transformation is rigid, then it is not strongly mixing. One might think to try to show that the set of rigid extensions of $T_0$ are a dense, $G_{\delta}$ set. However, first one must know what it means for an extension to be rigid.

\begin{defin} \label{Def:RigidExtension}
	Let $T \in \mathcal{G}_{T_0}$. We say that $T$ is \emph{rigid relative to} $T_0$, or $T$ is a \emph{rigid extension} of $T_0$ if there exists a subsequence $n_k$ such that for all $f,g \in L^2(Z|X),$
	
	$$\lim_{k \to \infty} \norm{\bbe(T^{n_k}f \cdot \overline{g} | X)- \bbe(T_0^{n_k}f \cdot \overline{g}|X)}_{L^2(X)} = 0.$$ 
\end{defin}

Let $\mathcal{R}_{T_0} \subset \mathcal{G}_{T_0}$ denote the set of rigid extensions of $T_0$. Note that if $X$ were a single point, then Proposition \ref{Prop:RigidEquivalence} shows that relative rigidity coincides with classical rigidity. Just as rigid transformations are mutually exclusive with strongly mixing transformations, we see that the same is true of rigid and strongly mixing extensions.

\begin{lem} \label{Lem:REImpliesNotSME}
	For all $T_0 \in \mathcal{G}(X), \mathcal{S}_{T_0} \subset \mathcal{R}^c_{T_0}.$
\end{lem}

\begin{proof}
	Let $T \in \mathcal{S}_{T_0}$ and let $A = X \times [0,1/2].$ Note then that $\bbe(\chi_A |X) \equiv 1/2$, so for all  $n,  \bbe(T_0^n \chi_A \cdot \chi_A|X) \equiv 1/2$ and $T_0^n \bbe(\chi_A |X) \bbe(\chi_A |X) \equiv 1/4.$ Therefore,
	
	$$\lim_{n \to \infty} \norm{\bbe(T^{n} \chi_A \cdot \chi_A | X)-\frac{1}{4}}_{L^2(X)} = 0.$$
	
	Subsequently for all subsequence $\{n_k\},$
	
	$$\lim_{k \to \infty} \norm{\bbe(T^{n_k} \chi_A \cdot \chi_A | X)- \frac{1}{2}}_{L^2(X)} \ge \frac{1}{4} > 0.$$ 
\end{proof}

In fact, we can say something even stronger. Letting $\mathcal{A}_{T_0}$ be as in the proof of Proposition \ref{Prop:SMECodense}, then it's easy to see that $\mathcal{R}_{T_0} \subset \mathcal{A}_{T_0}.$ Thus we have $\mathcal{S}_{T_0} \subset \mathcal{A}^c_{T_0} \subset \mathcal{R}^c_{T_0}.$

With the additional note that $\mathcal{R}_{T_0}$ is always nonempty by taking $T_0 \times \mathds{I}_Y,$ it seems that $\mathcal{R}_{T_0}$ has all the properties we would expect from a relativization of rigidity. Unfortunately, one property may fail. Consider $S \in \mathcal{G}_{\mathds{1}_X}.$ It may well fail to be the case that for $T \in \mathcal{R}_{T_0}, S^{-1}TS \in \mathcal{R}_{T_0}.$ Because of this, we cannot apply Theorem \ref{Thm:ConjugacyFixed}, and are unable to show $\mathcal{R}_{T_0}$ is dense. However, although we cannot show $\mathcal{R}_{T_0}$ is a dense, $G_{\delta},$ we can still show it is a $G_{\delta}.$ First, we prove the following technical lemma.

\begin{lem} \label{Lem:RigidExtensionL2Dense}
	Let $T \in \mathcal{G}_{T_0},$ and let $D \subset L^2(Z|X),$ with $\norm{\norm{f}_{L^2(Z|X)}}_{L^{\infty}(X)} \le 1$ for all $f \in D,$ such that $D$ is dense in the unit ball of $L^2(Z|X),$ with respect to the $L^2(Z)$ norm topology. Further suppose that there exists a subsequence $n_k$ such that for all $f_i,f_j \in D,$
	
	$$\lim_{k \to \infty} \norm{\bbe(T^{n_k}f_i \cdot \overline{f_j}|X) - \bbe(T_0^{n_k}f_i \cdot \overline{f_j}|X)}_{L^2(X)} = 0.$$
	
	Then $T \in \mathcal{R}_{T_0}.$
\end{lem}

\begin{proof}
	Fix $h,g$ in the unit ball in $L^2(Z|X)$ and let $\{h_j\}, \{g_j\} \subset D$ such that
	
	$$h_j \xrightarrow{L^2(Z)} h, g_j \xrightarrow{L^2(Z)} g.$$
	
	The proof of Lemma 6 in \cite{SchnurrWM} shows that (uniformly in $n$) $\bbe(T^nh_j \overline{g_j}|X) \to \bbe(T^nh \overline{g}|X),$ and $\bbe(T_0^n h_j \overline{g_j}|X) \to \bbe(T_0^nh \overline{g}|X)$ in $L^2(X)$. Now,
	
	\begin{align*}
	&\quad\ \norm{\bbe(T^n h \cdot \overline{g}|X) - \bbe(T_0^n h \cdot \overline{g}|X)}_{L^2(X)} \\
	&\le \ \norm{\bbe(T^n h \cdot \overline{g}|X) - \bbe(T^n h_j \cdot \overline{g_j}|X)}_{L^2(X)} \\
	&+ \ \norm{\bbe(T^n h_j \cdot \overline{g_j}|X) - \bbe(T_0^n h_j \cdot \overline{g_j}|X)}_{L^2(X)} \\
	&+ \ \norm{\bbe(T_0^n h_j \cdot \overline{g_j}|X) - \bbe(T_0^n h \cdot \overline{g}|X)}_{L^2(X)}.
	\end{align*}
	
	By hypothesis, there is a subsequence $\{n_k\}$ (independent of $j$) such that the middle term converges to 0. Further, the first and third terms converge to 0 as $j \to \infty$ uniformly in $n$, so 
	
	$$\lim_{k \to \infty} \norm{\bbe(T^{n_k}h \cdot \overline{g}|X) - \bbe(T_0^{n_k} h \cdot \overline{g}|X)}_{L^2(X)} = 0.$$
	
	By scaling this will also hold for all $h,g \in L^2(Z|X),$ so $T \in \mathcal{R}_{T_0}.$
\end{proof}

\begin{prop} \label{Prop:REGdelta}
	Let $T_0 \in \mathcal{G}(X)$ be fixed. Then $\mathcal{R}_{T_0}$ is a $G_{\delta}$ set.	
\end{prop}

\begin{proof}
	Let $\{f_i\}$ be an $L^2(Z)$-dense subset of $L^2(Z|X)$ and define
	
	$$U_{i,j,k,n} := \left\{T \in \mathcal{G}_{T_0} | \norm{\bbe(T^n f_i \cdot \overline{f_j}|X) -  \bbe(T_0^n f_i \cdot \overline{f_j}|X) }_{L^2(X)} < \frac{1}{k} \right\}.$$
	
	Lemma \ref{Lem:RigidExtensionL2Dense} shows that
	
	$$\mathcal{R}_{T_0} = \bigcap_{i,j,k} \bigcup_{n > k} U_{i,j,k,n},$$
	
	so we need only show $U_{i,j,k,n}$ are open. To this end, we instead show that the complement
	
	$$V(n,f,g,\epsilon) := \{S \in \mathcal{G}_{T_0} : \norm{\bbe(S^nf \cdot \overline{g} |X) - \bbe(T_0^n f \cdot \overline{g} | X)} \ge \epsilon \}$$
	
	is closed. Let $(S_m) \subset V(n,f,g,\epsilon)$ be a sequence of Koopman operators with $(S_m)$ converging weakly to a Koopman operator $S$. Note that this implies that $S_m \to S$ strongly (as Koopman operators are all isometries). 
	
	We claim that if we have functions $g,h,h_1, h_2, \ldots \in L^2(Z|X),$ and $h_m \to h$ in $L^2(Z),$ then $\bbe(h_mg|X) \to \bbe(hg|X)$ in $L^2(X).$ Indeed, by Proposition 4 in \cite{SchnurrWM},
	
	\begin{align*}
	\norm{\bbe(g \cdot h_m|X) - \bbe(g \cdot h|X)}_{L^2(X)} &= \norm{\bbe(g \cdot (h-h_m)|X)}_{L^2(X)} \\
	&\le \norm{\norm{g}_{L^2(Z|X)}}_{L^{\infty}(X)} \norm{h_m-h}_{L^2(Z)}.
	\end{align*}
	
	Lastly note that if $S_m \to S$ strongly, then $S_m^n \to S^n$ strongly.
	
	With these facts, we see that $\bbe(S_m^nf \cdot \overline{g} |X) - \bbe(T_0^n f \cdot \overline{g} | X)$ converges to $\bbe(S^nf \cdot \overline{g} |X) - \bbe(T_0^n f \cdot \overline{g} | X)$ strongly. Thus, $S \in V(n,f,g,\epsilon),$ and so $V(n,f,g,\epsilon)$ is closed.
\end{proof}

\nocite{GlasnerWeiss}
\nocite{SchnurrWM}

\bibliographystyle{plain}
\bibliography{StrongMixingExtensions} 

\end{document}